\documentclass[reqno, a4paper]{amsart} 

\usepackage[latin1]{inputenc}
\usepackage{subfigure}

\usepackage{amssymb} 
\usepackage{amsthm}
\usepackage{upref} 
\usepackage{graphicx}
\usepackage{color}
\usepackage{url}
\usepackage{paralist} 
\usepackage{etoolbox} 
\usepackage{tikz}
\usetikzlibrary{decorations.pathmorphing}
\usepackage{blkarray} 
\usepackage{ifthen}

\usepackage[british]{babel}

\usepackage{calc}
 
 \usepackage[all]{xy}

\usepackage{mathrsfs} 

\usepackage[all]{xy}
 
\definecolor{darkcellcolour}{rgb}{0.75,0.75,0.75}

\usepackage{hyperref}

\hypersetup{pdfauthor={Matthias Lenz},pdftitle={}, 
  pdfkeywords={} }

\newtheorem{Theorem}{Theorem}
\newtheorem{Corollary}[Theorem]{Corollary}
\newtheorem{Proposition}[Theorem]{Proposition}
\newtheorem{Lemma}[Theorem]{Lemma}

\theoremstyle{definition}

\newtheorem{Example}[Theorem]{Example}

\theoremstyle{remark}

\DeclareMathOperator{\Gr}{Gr} 
\DeclareMathOperator{\GL}{GL}

\newcommand{\abs}[1]{\left|#1\right|}
\DeclareMathOperator{\rank}{rk}

\newcommand{\st}{s.\,t.\ } 
\newcommand{\ie}{\textit{i.\,e.\ }} 
\newcommand{\eg}{\textit{e.\,g.\ }} 

\newcommand{\N}{\mathbb{N}}
\newcommand{\Z}{\mathbb{Z}}

\newcommand{\R}{\mathbb{R}}
\newcommand{\CC}{\mathbb{C}}

\newcommand{\K}{\mathbb{K}}
\newcommand{\Acal}{\mathcal{A}}

\newcommand{\Ccal}{\mathcal{C}}

\newcommand{\Tcal}{\mathcal{T}}
\newcommand{\Rcal}{\mathcal{R}}

\newcommand{\Gcal}{\mathcal G}

\author{Matthias Lenz} 
\address{Universit\'e de Fribourg, D\'epartement de Math\'ematiques, 1700 Fribourg, 
Switzerland}
\email{maths@matthiaslenz.eu}
\thanks{%
The author
was supported by a
fellowship within the postdoc programme of the German Academic Exchange Service (DAAD)%
}

\title[Representations of weakly multiplicative arithmetic matroids] 
{%
 Representations of weakly multiplicative arithmetic matroids are unique
}

\date{\today}
\address{%
}


\begin{document}

\begin{abstract}
   An arithmetic matroid is weakly multiplicative if the multiplicity of at least one of its bases is equal to the product of the multiplicities of  its elements. 
  We show that if such an  arithmetic matroid can be represented by an integer matrix, then this matrix is uniquely determined.
  This implies that the integer cohomology ring of a centred toric arrangement whose arithmetic   matroid is weakly multiplicative
  is determined by its poset of layers. 
  This partially answers a question asked by Callegaro--Delucchi.
\end{abstract}

\maketitle

\section{Introduction}

 An arithmetic matroid $\Acal$ is a triple $(E, \rank, m)$, where $(E,\rank)$ is a matroid on the ground set $E$ with rank function $\rank$ 
 and $m: 2^E\to \Z_{\ge 1}$ is the so-called multiplicity function
 \cite{branden-moci-2014,moci-adderio-2013}.
 In the representable case, \ie when the 
 arithmetic matroid is determined by a list of integer vectors,
 this multiplicity function records data such as the 
 absolute value of the determinant of a basis.

 Arithmetic matroids were recently introduced by D'Adderio and Moci 
 \cite{moci-adderio-2013}.
 They capture many combinatorial and topological properties of toric arrangements \cite{callegaro-delucchi-2017,lawrence-2011,moci-tutte-2012}
 in a similar way as matroids carry information about the corresponding hyperplane arrangement \cite{orlik-terao-1992,stanley-2007}.
 The study of arithmetic matroids can be seen as a step towards 
 the development of combinatorial frameworks to study the 
 topology of  very broad classes of spaces that are   
 complements
 of normal crossing divisors in smooth projective varieties.
 See the introduction of \cite{callegaro-delucchi-2017} for more details on this line of research.
 Toric arrangements and arithmetic matroids play an important role in the theory of vector
 partition functions, which describe the number of integer points in polytopes
 \cite{concini-procesi-book,lenz-arithmetic-2016}.
 They also appear naturally in the study    
 of cell complexes and Ehrhart theory of zonotopes 
 \cite{bajo-burdick-chmutov-2014,moci-tutte-2012,stanley-1991}.
 \smallskip

 Let $X\in \Z^{d\times N}$ be a matrix. The arithmetic matroid represented by $X$ is invariant under
 a left action of $\GL(d,\Z)$ on $X$ and under multiplication of some of the columns by $-1$. Therefore, when we are saying that a representation is 
 unique, we mean that any two distinct representations are equal up to these two types of transformations.
 
  An arithmetic matroid is \emph{torsion-free} if $m(\emptyset)=1$. 
 Let $\Acal=(E,\rank, m)$ be a torsion-free arithmetic matroid.
 Let $B\subseteq E$ be a basis. We say that $B$  is \emph{multiplicative} 
 if it  satisfies $m(B) =  \prod_{x\in B} m(\{ x \})$.
 This condition is always satisfied if $m(B)=1$.
     We call a torsion-free arithmetic matroid \emph{weakly multiplicative}
     if it has at least one multiplicative basis. This notion was introduced in \cite{lenz-ppcram-2017}.

 \begin{Theorem}%
    \label{Theorem:RepresentationUniqueness}
    Let $\Acal=(E,\rank,m)$ be an arithmetic matroid of rank $d$ that is weakly multiplicative, torsion-free, and representable.
    Then $\Acal$ has a unique representation, \ie
    if $X \in \Z^{d\times N}$ and $X' \in\Z^{d\times N}$ both represent $\Acal$, then there is a matrix
    $T\in \GL(d,\Z)$  and a  diagonal matrix   $D\in \Z^{N\times N}$ with diagonal entries
    in $\{1,-1\}$ \st
    $X' = T X  D$.
 \end{Theorem}
 Callegaro and Delucchi have recently put forward an incorrect proof\footnote{%
 In the proof in \cite{callegaro-delucchi-2017}, the argumentation in case b) 
 is flawed.
 For example, the proof    fails for the matrix
 $X = \left(\begin{array}{rrrrrr}
   1 & 0 & 0 & 1 & 0 & 1 \\
   0 & 1 & 0 & 1 & 1 & 0 \\
   0 & 0 & 1 & 0 & 1 &-1 \\
  \end{array}\right)$.
  In the inductive step, it is claimed that the bottom right entry of $X$ can be made positive, while all other signs are preserved (case b).
  This is false. }
 of this theorem in the special case where one basis has multiplicity $1$ 
 \cite[Theorem~7.2.1]{callegaro-delucchi-2017}.
 \smallskip

 Let $X\in \Z^{d\times N}$. Each column of $X$ defines a character $\chi : (\CC^*)^d \to \CC^*$ of the complex torus $ (\CC^*)^d$. 
 The set of kernels of these characters is called the centred toric arrangement defined by $X$.
  Callegaro and Delucchi asked whether the isomorphism type of 
  the integer cohomology ring of the complement of a complexified toric arrangement is determined combinatorially, \ie by the poset of layers of the toric arrangement
 \cite{callegaro-delucchi-2017}.
 Since the poset of layers encodes the arithmetic matroid \cite[Lemma~5.4]{moci-tutte-2012}, Theorem~\ref{Theorem:RepresentationUniqueness} implies an affirmative answer in the  
 special case of centred toric arrangements whose arithmetic matroid is weakly multiplicative.

  The condition that the arithmetic matroid of the arrangement is weakly multiplicative can also be explained geometrically. 
  Let $\Tcal_X = \{ \chi_1^{-1}(1), \ldots, \chi_N^{-1}(1)  \}$  be a centred toric arrangement, where each $\chi_i$ denotes a character.
  We assume that $\Tcal_X$ is essential, \ie  $ \bigcap_{i=1}^N \chi_i^{-1}(1) $ is $0$-dimensional.
  Then the arithmetic matroid corresponding to $\Tcal_X$ is weakly multiplicative if and only if the following condition is satisfied:
  there is a set $I\subseteq [N]$ of cardinality $d$ \st 
  $\bigcap_{i \in I} \chi_i^{-1}(1)$ is $0$-dimensional (\ie $I$ is a basis of the corresponding matroid) and 
  the number of connected components of the intersection $\bigcap_{i\in I}  \chi_i^{-1}(1)$ is equal to the product of the numbers of connected components
  of the   $\chi_i^{-1}(1)$ for $i\in I$.

  \begin{Corollary}
    Let $\Tcal_X$ be a centred toric arrangement   in $(\CC^*)^d$ whose corresponding arithmetic matroid is weakly multiplicative. 
    Then the integer cohomology ring of  $\Tcal_X$ is determined by its poset of layers.
  \end{Corollary}

  This result is a step towards 
  a better understanding of one of the main problems in arrangement theory:
  to what extent 
  is the topology of the complement of the arrangement
  determined by the combinatorial data?

 \smallskip
 The following example shows that the condition in Theorem~\ref{Theorem:RepresentationUniqueness} that the arithmetic matroid is weakly multiplicative is necessary.
\begin{Example}
 For $a,b\in \Z$, we define the matrix
 \begin{equation}
 X_{a,b} :=
  \begin{pmatrix}
      1 & a \\
      0 & b 
  \end{pmatrix}.
 \end{equation}
 Let $b\ge 2$.
 Then for any $a\in [b-1]$ that is relatively prime to $b$, the matrix
 $X_{a,b}$ is in Hermite normal form and it represents an arithmetic matroid $\Acal_b$ that is independent of $a$.
$\Acal_b$ is the arithmetic matroid with underlying uniform matroid $U_{2,2}$, whose multiplicity function is equal to $b$ on the  whole ground set 
 and $1$ otherwise.
 \end{Example}

\section{Background}
  
  \subsection{Notation}
  We will  use capital letters to denote matrices and the corresponding small letters to denote their entries.
  For $N\in \N$, we will write $[N]$ to denote the set $\{1,\ldots,N\}$.
  Usually, $N$ will denote the cardinality of a set and $d$ the dimension of the ambient space. We will always assume $d\le N$.

  \subsection{Arithmetic matroids}
  We assume that the reader is familiar with the basic notions of matroid theory \cite{MatroidTheory-Oxley}.
  An arithmetic matroid is a triple $(E, \rank, m)$, where
 $(E, \rank)$ is a matroid and $m : 2^E \to \Z_{\ge 1}$ denotes the
 \emph{multiplicity function}, that satisfies certain axioms.
   Since we are only discussing representable arithmetic matroids in this note, we do not 
  list the axioms for the multiplicity function of an arithmetic matroid here.
  They  can be found in \cite{branden-moci-2014}.

  A \emph{representable arithmetic matroid} is an arithmetic matroid that can be represented by
 a finite list of elements of a finitely generated abelian group $ G \cong \Z^d \oplus \Z_{q_1} \oplus \ldots \Z_{q_n}$.
 Representable and torsion-free arithmetic matroids can be represented by a finite list of elements of a lattice $G \cong \Z^d$.
 We will only consider this type of arithmetic matroid.
 We will assume that the ground set is always $E = \{e_1,\ldots, e_N\}$.
 Then a list $X$ of $N$ vectors in $\Z^d$ can be identified with the matrix $X\in \Z^{d\times N}$ whose columns are the entries of the list.
 
 A list of vectors $X = (x_e)_{e\in E} \subseteq\Z^d$ represents a vectorial matroid $(E,\rank)$ in the usual way.
 The multiplicity function $m$ defined by $X$ 
 is defined as $m(S) := \abs{ (\left\langle  S\right\rangle_\R \cap \Z^d) / \left\langle S\right\rangle}$ for $S\subseteq E$.
 Here, $\langle S \rangle\subseteq \Z^d$ denotes the subgroup generated  by $\{x_e: e\in S\}$ and $\langle S \rangle_\R\subseteq \R^d$
 denotes the subspace spanned by the same set.
 We will write  $\Acal(X)$  to denote the arithmetic matroid   that is represented by $X$.
 
 Let $X\in \Z^{d\times N}$ and let $B\in \Z^{d\times d}$ be a submatrix of full rank. Slightly abusing notation, we will also write
 $B$ to denote the corresponding basis of the underlying arithmetic  matroid.
 It is well-known (\eg it is a special case of \cite[Theorem~2.2]{stanley-1991})
 that
 \begin{equation}
 \label{eq:BasisMultiplicityDeterminant}
    m(B) = \abs{ \det(B) }. 
 \end{equation}

    It follows  from the definition that for $X\in \Z^{d\times N}$, $T\in \GL(d,\Z)$, and $D\in \Z^{N\times N}$ a diagonal matrix
    whose diagonal entries are contained in $\{1,-1\}$, the matrices $X$ and $T \cdot X \cdot D$ represent the same arithmetic matroid.
    In other words, applying a unimodular transformation from the left and multiplying some columns by $-1$ 
    does not change the arithmetic matroid that is represented
    by a matrix.

\subsection{Hermite normal form}
\label{Subsection:HNF}
 We say that matrix $X\in \Z^{d\times N}$ of full rank $d\le N$ is in \emph{Hermite normal form} 
 if for all $i \in [d]$, $0 \le x_{ij} < x_{jj}$ for $i<j$ and $x_{ij}=0$ for $i>j$, \ie
 the first $d$ columns of $X$ form an upper triangular matrix and the diagonal elements are strictly
 bigger than the other elements in the same column.
 It is not completely trivial, but well-known,  that any matrix $X \in \Z^{d\times N}$ of full rank $d$ can be brought into Hermite normal form
 by multiplying it from the left with a unimodular matrix $T\in \GL(d, \Z)$  if the first $d$ columns form a basis
 (\cite[Theorem~4.1 and Corollary~4.3b]{schrijver-TheoryLinearIntegerProgramming}).
 Since such a multiplication does not change the arithmetic matroid represented by the matrix, 
 we will be able to assume that a representation $X$ of a torsion-free arithmetic matroid $\Acal$ is in Hermite normal form. 

 We recall the following simple lemma:
\begin{Lemma}[\cite{lenz-ppcram-2017}]
   \label{Lemma:DiagonalMatrixMultiplicativeBasis}
   Let $X \subseteq \Z^d$ be a list of vectors and let $B$ be a multiplicative basis for the arithmetic matroid  $\Acal(X) = (E, \rank, m)$. 
   Let $X'$ denote the Hermite normal form of $X$ with respect to $B$.
   Then 
   the columns of $X'$ that correspond to $B$ form a diagonal matrix.
 \end{Lemma}

  \subsection{Toric arrangements}
  
  Let $T_\CC:=(\CC^*)^d$ be the \emph{complex} or \emph{algebraic torus} and let 
  $T_\R:= (S^1)^d$ be the \emph{real torus}.
  As usual, $S^1 := \{ z \in \CC : \abs{z} = 1 \}$. 
 Each $ \lambda=(\lambda_1,\ldots, \lambda_d)\in \Z^d$ determines a character of the torus, \ie a map
 $\chi_\lambda: T_\CC \to \CC^*$ (or $T_\R \to S^1$ in the real case) via
 $ \chi_\lambda( \phi_1, \ldots, \phi_d) := \phi_1^{\lambda_1}   \cdots   \phi_d^{\lambda_d} $.
 A \emph{complex toric arrangement} in $T_\CC$ is a finite set 
 $\Tcal=\{ T_1,\ldots, T_N \}$ with $T_i := \chi_i^{-1}(a_i)$, where $\chi_i$ is a character and $a_i\in \CC^*$ for all $i \in [N]$.
 A \emph{real toric arrangement} is defined similarly: in this case the $\chi_i$ are real characters and $a_i\in S^1$.
 A complex toric arrangement is called \emph{complexified} if all $a_i$ are contained in $S^1$.
 A toric arrangement is called \emph{centred} if $ a_i = 1 $ holds for all $i \in [N]$.
 The set of characters defining a toric arrangement in the $d$-dimensional torus 
 can be identified with a list of vectors in $\Z^d$. 
 The arithmetic matroid represented by this list of vectors is the arithmetic matroid corresponding to the toric arrangement.
A \emph{layer} of a toric arrangement $\Tcal$ is a connected component of a non-empty intersection  of elements of $\Tcal$.
 We obtain a poset structure on the set of layers of $\Tcal$ by ordering them by
 reverse inclusion, \ie $L \le L'$ if  $L' \subseteq L$.

\section{Proof}
 We will prove Theorem~\ref{Theorem:RepresentationUniqueness} 
 by carefully adapting and extending some methods that were developed by Brylawski and Lucas 
 in an article on uniquely representable matroids. 
 They showed that a
 representation $X$ of a matroid over some field $\K$ is unique (up to certain natural transformations)
 if the entries of $X$ are all contained in $\{0,1,-1\}$ \cite[Theorem~3.5]{brylawski-lucas-1976}. 

 Let $d\le N$ be two integers.
 Let $\Acal=(E,\rank,m)$ be an arithmetic matroid that is represented by a matrix
 $X\in \Z^{d\times N}$. Without loss of generality,  $E=[N]$.
 Let $B\subseteq E$ be a basis.
 We say that $X$ is in \emph{$B$-basic form} if there is a diagonal matrix 
 $B\in \Z^{d\times d}$ of full rank with non-negative entries and $A\in \Z^{ d \times (N-d) }$ \st
 $X = ( B \,|\, A )$.  Slightly abusing notation, we denote both the basis of $\Acal$ and the corresponding submatrix  by $B$.
 By definition, as a basis of $\Acal$, $B=[d]$.
 We will index the columns of $A$ by $d+1,\ldots, N$. 
 If $\Acal$ is a weakly multiplicative  arithmetic matroid that is represented by a matrix $X$,
 then we may assume 
 by  Lemma~\ref{Lemma:DiagonalMatrixMultiplicativeBasis}
 that $X$ is in $B$-basic form.
 
 Let $C$ denote the matrix that is obtained from $A$ by setting all   non-zero entries to $1$. This is called the \emph{$B$-fundamental circuit incidence matrix}.
 This name is justified as follows: if we label the rows of $C$ by $e_1,\ldots, e_d$ and the columns by $e_{d+1},\ldots, e_{N}$, an entry $c_{ij}$ of
 $C$ is equal to $1$
 if and only if $e_i$ is contained in the unique circuit contained in $B\cup \{e_j\}$, the so-called \emph{fundamental circuit} of $B$ and $e_j$.
 
 The matrix $C$ can also be seen as the adjacency matrix of a bipartite graph $\Gcal_A$ with vertex set $\{r_1,\ldots, r_d\} \cup \{ c_{d+1},\ldots, c_{N}\}$,
 where $r_i$ corresponds to the $i$th row and $c_j$ corresponds to the $j$th column.
 It will be important that one can identify an edge $\{r_i, c_j\}$ of $\Gcal_A$ with a non-zero entry $a_{ij}$ of $A$.
 A spanning forest in this graph will be called a \emph{coordinatizing path}.
 Let us fix a forest $F$ in $\Gcal_A$. Let $c$ be an edge that is not contained in $F$.  
 The fundamental circuit of $F$ and $c$ is called 
  a \emph{coordinatizing circuit} for $c$.
  Note that the graph $\Gcal_A$ has $N$ vertices.
  Let $\kappa(A)$ denote its number of connected components.
  It is easy to see that
  every coordinatizing path has cardinality $ N - \kappa(A) $.

 \newcommand{\hc}[1]{\colorbox{darkcellcolour}{#1}} %
 \newcommand{\nc}[1]{\colorbox{white}{#1}} %
\begin{Example}
\label{Example:Constructions}
 Note that the matrix $X\in \Z^{3\times 7}$ is in $B$-basic form.
 \begin{equation}
  X = 
   \begin{blockarray}{rrrrrrr}
   e_1  & e_2 & e_3 & e_4 & e_5 & e_6 & e_7  \\
  \begin{block}{(rrrrrrr)}
    1  &  0  & 0   & -4  &  0  &  3 &  0 \, \\
    0  &  2  & 0   &  1  &  2  &  0 & -2 \, \\
    0  &  0  & 3   &  0  &  1  & -1 & -1 \, \\
  \end{block}
  \end{blockarray}
\qquad
 C = 
\begin{blockarray}{rrrrr} %
 & e_4 & e_5 & e_6 & e_7 \\
 \begin{block}{r(rrrr)}
 e_1 &  \hc 1  &  \nc  0  & \nc  1  & \nc  0  \\
 e_2 &  \hc 1  &  \hc  1  & \nc  0  & \nc  1 \\ 
 e_3 &  \nc 0  &  \hc  1  & \hc  1  & \hc  1 \\ 
\end{block}
\end{blockarray}\,,
 \end{equation}
 $C$ is the adjacency matrix of the graph $\Gcal_A$ in Figure~\ref{Figure:BipartiteGraph}.
 The entries of $C$ that are highlighted define a coordinatizing path  which corresponds to the spanning forest $F$ of $\Gcal_A$.
 There are only two edges in the graph $\Gcal_A$ that are not contained in the spanning forest:
 $a_{16}$ and $a_{27}$.
 They define the coordinatizing circuits
 $\{a_{25},a_{35}, a_{27}, a_{37}\}$ and
 $\{a_{14}, a_{24}, a_{25}, a_{35}, a_{36}, a_{16}\}$.
 To simplify notation, we have described the edges of $\Gcal_A$ by the corresponding entries of $A$.

 Using the method described in the proof of Lemma~\ref{Lemma:CoordinatizingPathUniqueness}, we can obtain 
 a matrix $X'$ from $X$ where all the elements of the coordinatizing path are positive.
 We first pick a vertex in $\Gcal_A$ that has degree $1$ in the spanning tree, which we remove from the graph. Then we iterate this process until we obtain a graph
 that has no edges.
 This leads to the following sequence of vertices: 
 $r_1, c_4, r_2, c_5, c_6, c_7$. 
 We obtain the matrix $X'$
 by multiplying by $-1$ (in that order) column $7$, column $6$, and row $1$. 
 In three cases, the entry was already positive so no rescaling was necessary. 
 \begin{equation}
   X' = 
   \begin{blockarray}{rrrrrrr}
   e_1  & e_2 & e_3 & e_4 & e_5 & e_6 & e_7  \\
  \begin{block}{(rrrrrrr)}
    1  &  0  & 0   & 4  &  0  &  3 &  0 \, \\
    0  &  2  & 0   & 1  &  2  &  0 & 2 \, \\
    0  &  0  & 3   &  0  &  1  & 1 & 1 \, \\
  \end{block}
  \end{blockarray}
 \end{equation}

\end{Example}

 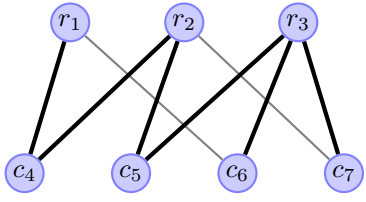
\begin{figure}[t]
 \begin{center}
  \begin{tikzpicture}[line join=round, scale=1,
 gitterpunkt/.style={color=black},
 fe/.style={color=black, ultra thick}, 
 ne/.style={color=gray, thick},  
 vertex/.style={circle,draw=blue!50,fill=blue!20,thick,
 inner sep=0pt,minimum size=5mm}
 ]

  \node [vertex] (r1) at (-1.5, 0) {$r_1$};
  \node [vertex] (r2) at (0, 0) {$r_2$};
  \node [vertex] (r3) at (1.5, 0) {$r_3$};

  \node [vertex] (c4) at (-2.1, -2) {$c_4$};
  \node [vertex] (c5) at (-0.7, -2) {$c_5$};
  \node [vertex] (c6) at ( 0.7, -2) {$c_6$};
  \node [vertex] (c7) at ( 2.1, -2) {$c_7$};

  \draw[fe]  (r1) -- (c4);
  \draw[ne]  (r1) -- (c6);
  \draw[fe]  (r2) -- (c4);
  \draw[fe]  (r2) -- (c5);
  \draw[ne]  (r2) -- (c7);
  \draw[fe]  (r3) -- (c5);
  \draw[fe]  (r3) -- (c6);
  \draw[fe]  (r3) -- (c7);

 \end{tikzpicture}
 \end{center}
 \caption{The bipartite graph corresponding to the matrix $C$ in Example~\ref{Example:Constructions}. The six edges contained in the  spanning forest $F$ are highlighted. }
 \label{Figure:BipartiteGraph}
\end{figure}

\begin{Lemma}%
 \label{Lemma:CoordinatizingPathUniqueness}
  Let $ X \in \Z^{d\times N} $ be a matrix of full rank $d$.
  Suppose that $X$ is in $B$-basic form, \ie there is a diagonal matrix $B\in \Z^{d\times d}$ of full rank with non-negative entries and $A\in\Z^{d\times (N-d)}$
  \st $X=(B\,|\,A)$.
  Let $P = \{ p_1,\ldots, p_{N-\kappa(A)}\}$ be a coordinatizing path and
  let $ \sigma\in   \{ -1, 1\}^{ N - \kappa(A) }$. 
  Then there is a   matrix $ X' = ( B\,|\,A' )$  that represents the same arithmetic matroid $\Acal(X)$ and 
  the entry $p_j$ of $A'$ is equal to $\sigma_j$ times the corresponding entry of $A$.
  The matrix $X'$ can be obtained from $X$ by a sequence of multiplications of rows and columns by $-1$.
\end{Lemma}
\begin{proof}
  This lemma is a modified version of \cite[Proposition~2.7.3]{brylawski-lucas-1976} and we are proving it in a similar way.
  The proof is by induction on $\abs{P}$. If $\abs{P}=0$, there is nothing to prove. Let us assume that we have proved that the 
  statement is true for all matrices $\tilde A$ that have a coordinatizing path $\tilde P$ with $\bigl|\tilde P\bigr|<k$.
  Suppose $\abs{P} = k \ge 1$. 
  Since every forest that contains at least one edge has a vertex of degree one, there is some $a_{ij} = p_s\in P$ which is the unique entry common to 
  $P$ and some line (row $r_i$ or column $c_j$) of $A$. Assume that line is row $r_i$. Then deleting that row from $A$ one easily sees that $\tilde P = P \setminus \{ p_s \}$ 
  is a coordinatizing path for the matrix 
  obtained from $A$ by deleting row $r_i$.
  By induction, we are able to change the signs of the entries of $ \tilde P $ as prescribed by $ \sigma $ by multiplying rows and columns by $-1$ 
  (which we may perform in $A$), giving $p_s=a_{ij}$ the value $ \tau a_{ij}$ for some $\tau\in \{-1,1\}$. 
  If we then multiply row $r_i$ in $A$ by $ \sigma_s \tau $, we assign $p_s$ the appropriate sign
  and we affect none of the entries of the coordinatizing path $\tilde P$ that were previously considered.

 Since multiplying rows and columns of a matrix $X$ by $-1$ does not change the arithmetic matroid $\Acal(X)$, both $X$ and $X'$ represent the same arithmetic matroid.
\end{proof}

 \begin{Lemma}
  \label{Lemma:RepresentationsUniqueUpToSign}
  Let $ X \in \Z^{d\times N} $ be a matrix of full rank $d$ that
  represents an arithmetic matroid $\Acal$.
  Suppose that $X$ is in $B$-basic form, in particular 
  $X=(B \,|\,A)$.
  Suppose that $X'=(B\,|\,A')\in \Z^{d\times N}$ represents the same arithmetic matroid.
  Then the entries of $A$ and $A'$ are equal up to sign, \ie $\abs{a_{ij}} = \abs{a_{ij}'}$.
 \end{Lemma}
 \begin{proof}
   Recall  that the columns of $ A $ and $ A' $ are labelled by $ d+1, \ldots, N$.
   For $j\in \{d+1,\ldots, N\}$, the set $\{e_1,\ldots, \hat e_i, \ldots, e_d, e_j\}$ is dependent if and only if 
   the determinant of the corresponding submatrices of $X$ and $X'$ is $0$.
   This holds if and only if $ a_{ij} = a'_{ij} = 0$.
   If the set is independent, \ie it is a basis, by \eqref{eq:BasisMultiplicityDeterminant}, %
   \begin{equation*}
    \label{eq:MatrixMultiplicityfunction}
       \abs{a_{ij}} = \frac{ m ( \{1,\ldots, \hat i, \ldots, d, j\} ) }{ \prod_{\nu\in [d] \setminus \{ i \} }{ b_{\nu\nu} } } = \abs{a_{ij}'}. \qedhere
   \end{equation*}
 \end{proof}
 \begin{Lemma}
 \label{Lemma:SubmatrixDeterminantHNF}
  Let $ X \in \Z^{d\times N} $ be a matrix of full rank $d$ that
  represents an arithmetic matroid $\Acal$.
  Suppose that $X$ is in $B$-basic form, 
  in particular $X=(B \,|\,A)$.
  Then up to sign, any non-zero subdeterminant of $A$ 
  is determined by the arithmetic matroid $\Acal(X)$.
 \end{Lemma}
 \begin{proof}
    Let  $I\subseteq [d]$ and $J\subseteq \{d+1,\ldots, N\}$ be two sets of the same cardinality.
  Let $S$ be the submatrix of $A$ whose rows are indexed by $I$ and whose columns are indexed by $J$.
  If $\det(S)\neq 0$, then $B' := ( [d] \setminus I) \cup J$ is a basis.  
 It follows   from \eqref{eq:BasisMultiplicityDeterminant}
 that
  $ m(B') = \abs{\det(S)} \prod_{\nu \in [d] \setminus I} b_{\nu\nu} $.
  Of course, $b_{\nu\nu }$ is equal to the multiplicity of the $\nu$th column of $B$.
 \end{proof}

\begin{Lemma}
 \label{Lemma:Uniqueness}
 The matrix $X'$ in 
 Lemma~\ref{Lemma:CoordinatizingPathUniqueness} is uniquely determined.
\end{Lemma}
\begin{proof}
  This proof uses some ideas of the proof of \cite[Theorem~3.2]{brylawski-lucas-1976}.
  Let $X''=(B\,|\,A'')$ be another matrix that satisfies the consequence
  in Lemma~\ref{Lemma:CoordinatizingPathUniqueness}.
  In particular, we assume that the entries of $A'$ and $A''$ in the coordinatizing path are equal.
  By Lemma~\ref{Lemma:RepresentationsUniqueUpToSign}, the entries of $A$, $A'$, and $A''$ must be equal up to sign.
  Hence it is sufficient to show that all non-zero entries of  $A'$ and $A''$ that are not contained in the
  coordinatizing path are equal.
  
  Recall that $C$ denotes the $B$-fundamental circuit incidence matrix.
  Let us consider  a  non-zero entry $\alpha$ of $C$ that is not contained in the coordinatizing path.
  $\alpha$ is contained in a unique coordinatizing circuit $\Ccal$.
  Let $a_1$ and $a_2$ denote the  entries of $A'$ and $A''$ that correspond to $\alpha$.
    
  Suppose first that  $\abs{\Ccal}=4$.
  Then $\Ccal$ corresponds to a $(2\times 2)$-submatrix of $A'$ or $A''$, respectively. 
  The three other entries besides $a_1$ or $a_2$
  are contained in the coordinatizing path $P$ and therefore, by assumption, they are equal for $A'$ and $A''$.
  We will denote these three entries by $b$, $c$, and $d$.
  Then (up to relabelling the entries), the determinants of the two submatrices are
   $ a_1d - bc$ and $ a_2 d - bc$, respectively.
  Since $X'$ and $X''$ define the same arithmetic matroid, it follows from Lemma~\ref{Lemma:SubmatrixDeterminantHNF}
  that the absolute values the two determinants must be equal.
  Now suppose $a_1 = -a_2$. Then  $\abs{ a_1d - bc } = \abs{ ((-a_1)d - bc }$ must hold. This is equivalent to
  $ a_1d - bc = -a_1d - bc$ or $ a_1d - bc = a_1d + bc $. Both cases are impossible if all four number are non-zero. Hence $a_1 = a_2$.
  
  Let $P_2$ be the union of the coordinatizing path $P$ with the coordinatizing circuit $\Ccal$.
   We have determined all the entries in $P_2$
  uniquely.
  Now by an analogous argument, we can uniquely determine all entries of $ C \setminus P_2$ which complete a circuit of size $4$ in $\Gcal_A$ with elements of $P_2$.
  Continuing this process we end by uniquely determining all entries which can be attained  by a sequence of circuits of size $4$, 
  three of whose members having been previously determined.
  We call the resulting set of determined entries $P_2^*$.

  Now let $\alpha \in C \setminus P_2^*$ be an entry that completes a circuit $\Ccal$ of 
  size $6$ in $\Gcal_A$ with elements of  $P_2^*$. 
  The circuit $\Ccal$ corresponds to a $3\times 3$ submatrix $S$ of $C$.
  Again, let $a_1$ and $a_2$ denote the entries of $A'$ and $A''$ that correspond to $\alpha$.
    There are two cases to consider:
  \begin{asparaenum}%
  \item  The $3\times 3$ submatrix $S$ has for its non-zero entries only the $6$ entries of $\Ccal$.
   In this case, $S$ has two non-zero entries in each row and column. Hence it is the sum of two permutation matrices.
   This implies that the corresponding subdeterminants of $A'$ and $A''$ are equal 
   to $a_1 x + y$ and $a_2 x + y$, respectively, for some $x, y \neq 0$.
   As above, it is easy to see that it is not possible to have $a_1=-a_2$  (which 
   implies
   $\abs{ a_1 x + y } = \abs{(-a_1)x+y} \,$)  if $ a_1, x, y \neq 0 $.
   \item If there is another non-zero entry $\beta$ in $S$, 
   then $\beta$ represents an additional edge which   short-circuits the circuit $\Ccal$ in the sense that it cuts across
   $\Ccal$ to 
   form a $\theta$-graph. Thus $\beta$ completes two smaller circuits with $ \Ccal \cup \{\beta\}$,
   one containing some previously determined elements and $\beta$, the other containing $\alpha$ and $\beta$.   The former circuit
   implies that $ \beta \in P_2^*$. Hence the latter circuit shows that $\alpha \in P_2^*$ as well.
   See Figure~\ref{Figure:BipartiteGraphSC} for an example of this setting.
   \end{asparaenum}

 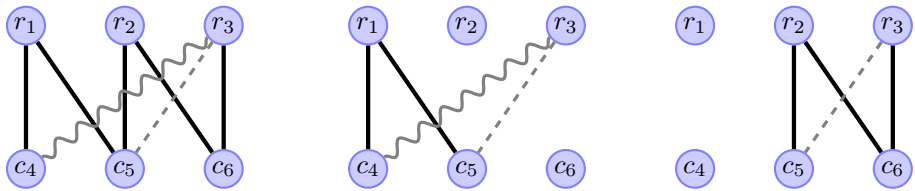
\begin{figure}[t]
 \begin{center}
  \begin{tikzpicture}[line join=round, scale=1,
 gitterpunkt/.style={color=black},
 fe/.style={color=black, ultra thick}, 
 ne1/.style={color=gray, very thick, decorate, decoration=snake },  
 ne2/.style={color=gray, very thick, dashed},  
 vertex/.style={circle,draw=blue!50,fill=blue!20,thick,
 inner sep=0pt,minimum size=5mm}
 ]

  \node [vertex] (r1) at (-1.3, 0) {$r_1$};
  \node [vertex] (r2) at (0, 0) {$r_2$};
  \node [vertex] (r3) at (1.3, 0) {$r_3$};

  \node [vertex] (c4) at (-1.3, -1.9) {$c_4$};
  \node [vertex] (c5) at (0, -1.9) {$c_5$};
  \node [vertex] (c6) at ( 1.3, -1.9) {$c_6$};

  \draw[fe]  (r1) -- (c4);
  \draw[fe]  (r1) -- (c5);
  \draw[fe]  (r2) -- (c5);
  \draw[fe]  (r2) -- (c6);
  \draw[fe]  (r3) -- (c6);
  \draw[ne1]  (r3) -- (c4);
  \draw[ne2]  (r3) -- (c5);

  \node [vertex] (xr1) at (3.2, 0) {$r_1$};
  \node [vertex] (xr2) at (4.5, 0) {$r_2$};
  \node [vertex] (xr3) at (5.8, 0) {$r_3$};
  \node [vertex] (xc4) at (3.2, -1.9) {$c_4$};
  \node [vertex] (xc5) at (4.5, -1.9) {$c_5$};
  \node [vertex] (xc6) at ( 5.8, -1.9) {$c_6$};

  \draw[fe]  (xr1) -- (xc4);
  \draw[fe]  (xr1) -- (xc5);
  \draw[ne1]  (xr3) -- (xc4);
  \draw[ne2]  (xr3) -- (xc5);

  \node [vertex] (yr1) at (7.5, 0) {$r_1$};
  \node [vertex] (yr2) at (8.8, 0) {$r_2$};
  \node [vertex] (yr3) at (10.1, 0) {$r_3$};
  \node [vertex] (yc4) at (7.5, -1.9) {$c_4$};
  \node [vertex] (yc5) at (8.8, -1.9) {$c_5$};
  \node [vertex] (yc6) at (10.1, -1.9) {$c_6$};

  \draw[fe]  (yr2) -- (yc5);
  \draw[fe]  (yr2) -- (yc6);
  \draw[fe]  (yr3) -- (yc6);
  \draw[ne2]  (yr3) -- (yc5);

 \end{tikzpicture}
 \end{center}
 \caption{A coordinatizing circuit is short-circuited as in the proof of Lemma~\ref{Lemma:Uniqueness}).
 The edges in $P_2^*$ are shown in black. The wavy edge is $\alpha$ and the dashed edge is $\beta$.
 }
 \label{Figure:BipartiteGraphSC}
\end{figure}

   We iterate the above argument to prove that the entries of  $A'$ and $A''$ that are contained in $P_3^*$ must be equal,  
   where $P_3^*$ denotes the set of all non-zero entries of $C$ which can be attained from $P$ by a sequence of circuits of size $2 t$ for $t  \le 3$.
   We define $P_k^*$ analogously and assume that we have uniquely determined all entries of $P_k^*$ for $ k < m $.
   If  $\alpha \in C \setminus P_{m-1}^*$ and $ \alpha $ completes a circuit $\Ccal$ of size $2m$ with entries from $P_{m-1}^*$ then there are two cases:
   \begin{asparaenum}[(1)]
   \item  The $m\times m$ submatrix $S$ of $C$ corresponding to the rows and columns of $\Ccal$ has no non-zero entries other than those of $\Ccal$.
   Then, as above, $S$ is the sum of two permutation matrices 
   and the corresponding subdeterminants of $A'$ and $A''$ are equal to $ a_1 x + y$ and $a_2x + y$ for some $x,y\neq 0$, which implies $a_1=a_2$.
   \item If $S$ contains another non-zero entry $\beta$ then $\Ccal \cup \{\beta\} $ is a $\theta$-subgraph of $\Gcal_A$.
   So using the same argument as in the $(3\times 3)$-case, by induction it follows that the entries of $A'$ and $A''$ that correspond to $\alpha$ must be equal.
   \qedhere
   \end{asparaenum}
\end{proof}

 \begin{proof}[Proof of Theorem~\ref{Theorem:RepresentationUniqueness}]
   By assumption, the matrices $X$ and $X'$ both have full rank.
   Let $B$ be a basis that is weakly multiplicative.
   By Lemma~\ref{Lemma:DiagonalMatrixMultiplicativeBasis}, we 
   may assume that both $X$ and $X'$ are in $B$-basic form, \ie
   $X = (B \,|\, A)$ and $X' = (B \,|\, A')$ for suitable matrices $A$, $A'$, and $B$.
   Lemma~\ref{Lemma:RepresentationsUniqueUpToSign} implies that the entries of $A$ and $A'$ must be equal up to sign.

   Let $P$ be a coordinatizing path.
   By Lemma~\ref{Lemma:CoordinatizingPathUniqueness}, we may assume that the entries of $A$ and $A'$ that are contained in $P$ are equal,
   after multiplying some rows and columns of $X$ by $-1$. We are permitted to do these operations:
  recall that multiplying a  row of $X$ by $-1$ corresponds to multiplying $X$ from the left with a certain matrix in $\GL(d,\Z)$.
  Multiplying a column of $X$ by $-1$ corresponds to multiplying $X$ from the right with a certain non-singular diagonal matrix with diagonal entries in
  $\{ 1, -1\}$. 
  We conclude by observing that Lemma~\ref{Lemma:Uniqueness} implies that all the remaining entries of $A$ and $A'$ must be equal too.
 \end{proof}

 \section{Arithmetic matroid strata of the integer Grassmannian}
 
 In this section we will use the results in this paper  to describe certain ``strata'' of an integer analogue of the Grassmannian.
 
  Recall that for a matrix $A\in \Z^{d \times (N-d)}$, $\kappa(A)$ denotes the  number of connected components of the bipartite graph with adjacency matrix $A$.
  We obtain the following result by combining 
  Lemma~\ref{Lemma:CoordinatizingPathUniqueness},
  Lemma~\ref{Lemma:RepresentationsUniqueUpToSign},
  and Lemma~\ref{Lemma:Uniqueness}.
 
 \begin{Proposition}
 \label{Proposition:NoOfRepresentations}
   Let $X=(B\,|\,A)\in \Z^{d\times N}$ with  $B\in \Z^{d\times d}$ a diagonal matrix of full rank $d$ and $A\in \Z^{ d \times (N-d)}$.
   Let $P$ be a coordinatizing path.
   For each of the $2^{N-\kappa(A)}$ possible choices of signs of the entries of $P$, there is a unique matrix $X_\sigma = (B \,|\, A_\sigma)$
   with these signs 
   that represents the same arithmetic matroid   $\Acal(X)$.    
   
   All representations of $\Acal(X)$ that are in $B$-basic form for this basis $B$ can be obtained in this way.
 \end{Proposition}

 Grassmannians are 
 fundamental objects in algebraic geometry (\eg \cite{OrientedMatroidsBook,harris-1995}).
 For a field $\K$, the Grassmannian $\Gr_\K(d,N)$ can be defined as the set of $(d\times N)$-matrices over $\K$ of full rank
 modulo a left action of $\GL(d,\K)$. Similarly, one can define the integer Grassmannian  $\Gr_\Z(d,N)$ as
  the set of all matrices $X\in \Z^{d\times N}$ of full rank, modulo a left action of $\GL(d,\Z)$.
 The set of representations of a fixed  torsion-free arithmetic matroid $\Acal$ of rank $d$ on $N$ elements
 is a subset of $ \Z^{d\times N}  $ that is invariant under a left action of
 $ \GL(d, \Z) $  and a right action of diagonal $(N\times N)$-matrices with entries in $\{\pm 1\}$, \ie of
 $(\Z^*)^N$, the maximal multiplicative subgroup of $\Z^N$.
  This leads to a stratification of the integer Grassmannian $\Gr_\Z(d,N)$ into arithmetic matroid strata $\Rcal(\Acal)=
  \{ \bar X \in \Gr_\Z(d,N) : X \text{ represents } \Acal \}$.
 Proposition~\ref{Proposition:NoOfRepresentations} allows us to calculate the cardinality of certain arithmetic matroid strata.
 \begin{Corollary}
 Let $\Acal$ be an arithmetic matroid of rank $d$ on $N$ elements that is weakly multiplicative and representable.
 Let $X=(B\,|\,A)$ be a representation in $B$-basic form.
  Then the arithmetic matroid stratum of $\Acal$  of the integer Grassmannian   $\Gr_\Z(d,N)$ has $2^{ N - \kappa( A ) }$ elements.
 \end{Corollary}

\subsection*{Acknowledgements}
 The author would like to thank Elia Saini for several interesting discussions.

\renewcommand{\MR}[1]{} 

 \bibliographystyle{amsplain}

\providecommand{\bysame}{\leavevmode\hbox to3em{\hrulefill}\thinspace}
\providecommand{\MR}{\relax\ifhmode\unskip\space\fi MR }
\providecommand{\MRhref}[2]{%
  \href{http://www.ams.org/mathscinet-getitem?mr=#1}{#2}
}
\providecommand{\href}[2]{#2}

\end{document}